\documentclass{article}
\usepackage{amsmath}

\usepackage{ yhmath,amsfonts,amssymb,amsthm,bbding}

\newcommand{\counte}{theorem}
\newtheorem{theorem}{\bf Theorem}[section]

\newtheorem{defn}[\counte]{\bf Definition}

\newtheorem{question}[\counte]{\bf Question}

\newtheorem{lemma}[\counte]{\bf Lemma}

\newtheorem{remark}[\counte]{\bf Remark}

\allowdisplaybreaks

\numberwithin{equation}{section}

\renewcommand{\thefootnote}{\fnsymbol{footnote}}

\begin{document}

\renewcommand{\thefootnote}{\arabic{footnote}}

\centerline{\bf\Large Pythagorean Theorem}
\centerline{\bf\Large \& curvature with lower or upper bound \footnote{Supported by NSFC 11971057 and BNSF Z190003. \hfill{$\,$}}}

\vskip5mm

\centerline{Xiaole Su, Hongwei Sun, Yusheng Wang\footnote{The
corresponding author (E-mail: wyusheng@bnu.edu.cn). \hfill{$\,$}}}

\vskip6mm

\noindent{\bf Abstract.} In this paper, we give a comparison version of Pythagorean Theorem
to judge the lower or upper bound of the curvature of Alexandrov spaces (including Riemannian manifolds).

\vskip1mm

\noindent{\bf Key words.} Pythagorean Theorem, Alexandrov space, Toponogov's Theorem

\vskip1mm

\noindent{\bf Mathematics Subject Classification (2000)}: 53-C20.

\vskip6mm

\setcounter{section}{-1}


\section{Introduction}

\vskip2mm

Let $\mathbb S^n_k$ be the complete and simply
connected $n$-dimensional space form with constant curvature $k$. For any
minimal geodesics $[pq],[pr]\subset \mathbb S^n_k$ which form an angle $\angle qpr$ at $p$,
the Law of Cosine says that
$$
\begin{cases}
\begin{aligned}
\cos(\sqrt{k}|qr|) & =\cos(\sqrt{k}|pq|)\cos(\sqrt{k}|pr|)+\sin(\sqrt{k}|pq|)\sin(\sqrt{k}|pr|)\cos\angle qpr, & k>0 \\
|qr|^2  & =|pq|^2+|pr|^2-2|pq||pr|\cos\angle qpr, & k=0 \\
\cosh(\sqrt{-k}|qr|) & =\cosh(\sqrt{-k}|pq|)\cosh(\sqrt{-k}|pr|)-\sinh(\sqrt{-k}|pq|)\sinh(\sqrt{-k}|pr|)\cos\angle qpr,  & k<0	
\end{aligned}\ ,\end{cases}
$$
where $|\cdot\cdot|$ denotes the distance between two given points. In particular, if $\angle qpr=\frac\pi2$, then
\begin{equation}\label{eqn0.1}
\begin{cases}
\begin{aligned}
\cos(\sqrt{k}|qr|) & =\cos(\sqrt{k}|pq|)\cos(\sqrt{k}|pr|), & k>0 \\
|qr|^2 & =|pq|^2+|pr|^2, & k=0 \\
\cosh(\sqrt{-k}|qr|) & =\cosh(\sqrt{-k}|pq|)\cosh(\sqrt{-k}|pr|),  & k<0	
\end{aligned}\ ,
\end{cases}
\end{equation}
which is the famous Pythagorean Theorem on $\mathbb S^n_k$, especially the middle one for $k=0$
(the Gougu Theorem in China).
A fascinating thing is that the Law of Cosine can be derived from Pythagorean Theorem, i.e. the Law of Cosine
is equivalent to Pythagorean Theorem.

For a general Riemannian manifold $M$, it is well-known that a necessary and sufficient condition of sectional curvature $\sec_M\geq k$ (or $\leq k$)
is a local comparison version of the Law of Cosine. Namely,

\begin{theorem} \label{thm0.1}
Let $M$ be a complete Riemannian manifold. Then
$\sec_M\geq k\ (\leq k)$  if and only if for any $x\in M$ there exists a neighborhood $U_x$ of $x$ such that for any minimal geodesics $[pq],[pr]\subset U_x$
$$
\begin{cases}
\begin{aligned}
\cos(\sqrt{k}|qr|) & \geqslant (\leqslant)\ \cos(\sqrt{k}|pq|)\cos(\sqrt{k}|pr|)+\sin(\sqrt{k}|pq|)\sin(\sqrt{k}|pr|)\cos\angle qpr, & k>0 \\
|qr|^2&  \leqslant (\geqslant)\ |pq|^2+|pr|^2-2|pq||pr|\cos\angle qpr, & k=0 \\
\cosh(\sqrt{-k}|qr|)& \leqslant (\geqslant)\ \cosh(\sqrt{-k}|pq|)\cosh(\sqrt{-k}|pr|)-\sinh(\sqrt{-k}|pq|)\sinh(\sqrt{-k}|pr|)\cos\angle qpr,  & k<0	
\end{aligned}\ ;\end{cases}
$$
and equality holds for all $x\in M$ and all $[pq],[pr]\subset U_x$ if and only if $\sec_M\equiv k$.
\end{theorem}

Inspired by the relation between Pythagorean Theorem and the Law of Cosine, a natural question is:

\begin{question} \label{qes0.2}
{\rm Is there a comparison version of Pythagorean Theorem to judge the lower or upper bound of $\sec_M$?}
\end{question}

In this paper, for three points $p,q,r$ in a metric space, we denote by $\tilde\angle_{k} qpr$ the angle between $[\tilde p\tilde q]$ and $[\tilde p\tilde r]$ in $\mathbb S^2_{k}$ with $|\tilde p\tilde q|=|pq|$, $|\tilde p\tilde r|=|pr|$ and $|\tilde q\tilde r|=|qr|$. Note that a comparison version of (\ref{eqn0.1}),
$$
\begin{cases}
\begin{aligned}
\cos(\sqrt{k}|qr|) &\geq\ (\leq)\ \cos(\sqrt{k}|pq|)\cos(\sqrt{k}|pr|), & k>0 \\
|qr|^2 & \leq\ (\geq)\ |pq|^2+|pr|^2, & k=0 \\
\cosh(\sqrt{-k}|qr|) & \leq\ (\geq)\ \cosh(\sqrt{-k}|pq|)\cosh(\sqrt{-k}|pr|),  & k<0
\end{aligned}\ ,
\end{cases}
$$
is equivalent to
$$\tilde\angle_kqpr\leq\ (\geq)\ \frac\pi2.$$

The main goal of the paper is to give a positive answer to Question 0.2 not only for a Riemannian manifold but also for an Alexandrov space
through the following result.

\vskip2mm

\noindent {\bf Theorem A.} {\it Let $X$ be a complete Alexandrov space. Then $X$ is of curvature
$\geq\ (\leq)\ k$ if and only if for any $x\in X$ there exists a neighborhood $U_x$ of $x$
such that, for any $q\in U_x$ and $[r_1r_2]\subset U_x$, if there is $p\in [r_1r_2]^\circ$ (the interior part of $[r_1r_2]$)
satisfying $|qp|=|q[r_1r_2]|$ (where $|q[r_1r_2]|\triangleq\min\limits_{s\in [r_1r_2]}\{|qs|\}$) then
\begin{equation}\label{eqn0.2}
\tilde\angle_kqpr_i\leq\ (\geq)\ \frac\pi2,\ i=1,2.
\end{equation}
Moreover, if equality in {\rm (\ref{eqn0.2})} holds for all $x\in X$ and all $q\in U_x$, $[r_1r_2]\subset U_x$ and such $p\in[r_1r_2]$, then
$X^\circ$ (the interior part of $X$) is a Riemannian manifold with sectional curvature equal to $k$. }

\vskip2mm

In this paper, that $X$ is an Alexandrov space means that for any $x\in X$ there is a real number $k_x$ such that a neighborhood of $x$
is a so-called Alexandrov space with curvature $\geq k_x$ or $\leq k_x$; and we call $x$ a CBB-type or CBA-type point when
$X$ is of curvature $\geq k_x$ or $\leq k_x$ around $x$ respectively. Of course, a Riemannian manifold is an Alexandrov space. Note that although $X$ is complete,
$X$ might not be equal to $X^\circ$ (this differs from the Riemannian case). Refer to Section 1 for details on Alexandrov spaces.

Note that if $X$ is a complete CBB-type Alexandrov space, then `$|qp|=|q[r_1r_2]|$' in Theorem A implies that
$\angle qpr_i=\frac\pi2$ for any $[qp]$ and $[pr_i]$ (see Lemma \ref{lem1.6} below). Thereby, it is clear
that Theorem A has the following corollary, a positive answer to Question \ref{qes0.2}.

\vskip2mm

\noindent {\bf Corollary B.} {\it Let $X$ be a complete CBB-type Alexandrov space. Then $X$ is of curvature
$\geq (\leq)\  k$  if and only if for any $x\in X$ there exists a neighborhood $U_x$ of $x$
such that, for all $[pq],[pr]\subset U_x$ with $\angle qpr=\frac\pi2$,
\begin{equation}\label{eqn0.3}
\tilde\angle_kqpr\leq\ (\geq)\ \frac\pi2.
\end{equation}
Moreover, if equality in {\rm(0.3)} holds for all $x\in X$ and all such $[pq],[pr]\subset U_x$, then
$X^\circ$ is a Riemannian manifold with sectional curvature equal to $k$.}

\begin{remark}\label{rem0.3} {\rm
For the rigidity part of Corollary B, one can consider the following simple example. Note that a geodesic triangle on $\mathbb S_k^2$ separates $\mathbb S_k^2$ into two parts with boundary.
The smaller one is  a complete Alexandrov space with curvature $\geq k$,
but not a Riemannian manifold with boundary, and satisfies Pythagorean Theorem locally.
(However, the larger one is an Alexandrov space with curvature $\leq k$, and does not satisfy Pythagorean Theorem
around the vertices of the triangle.)}
\end{remark}

\begin{remark}\label{rem0.4} {\rm
For a CBA-type Alexandrov space $X$,
we cannot judge whether $X$ is of curvature $\geq k$ or $\leq k$ in a similar way as Corollary B. For example,
the union of three rays in $\mathbb R^2$ starting from a common point (with the induced intrinsic metric),
a CBA-type Alexandrov space, has no $\frac\pi2$-angle nor lower curvature bound.}
\end{remark}

\begin{remark}\label{rem0.5} {\rm
If $X$ is a Riemannian manifold, one can give a proof for Theorem A via the second variation formula and the comparison results
on index forms (the main tools in proving the well-known Rauch's Theorem),
which do not work when $X$ is a general Alexandrov space. Of course, in our proof
relying on Toponogov's Theorem, many arguments can be removed
in the case where $X$ is a Riemannian manifold (i.e. the proof can be much shorter).}
\end{remark}

As an application of Theorem A, we supply a way to judge whether a point in a CBB-type Alexandrov space is a Riemannian one (i.e. its space of directions is a unit sphere, see Section 1).

\vskip2mm

\noindent {\bf Theorem C.} {\it Let $x$  be an interior point in a complete CBB-type Alexandrov space.
If there is a function $\chi(\varepsilon)$ with $\varepsilon>0$ and $\chi(\varepsilon)\to 0$ as $\varepsilon\to0$
such that, for all $[pq],[pr]\subset B_x(\varepsilon)$ with $\angle qpr=\frac\pi2$,
$$
\left|\frac{|qr|^2}{|pq|^2+|pr|^2}-1\right|< \chi(\varepsilon),
$$
then $x$ is a Riemannian point. }

\vskip2mm

We would like to point out that the condition for $x$ to be a Riemannian point in Theorem C should just be sufficient, but not necessary.

As an almost immediate corollary of Theorem C, we have the following known result (\cite{AKP}).

\vskip2mm

\noindent {\bf Corollary D.} {\it Let $x$  be an interior point in a complete CBB-type Alexandrov space $X$.
If in addition $x$ is a CBA-type point, then $x$ is a Riemannian point
(as a result, if each point in $X^\circ$ is a CBA-type one, then $X^\circ$ is a manifold). }

\vskip3mm

In the rest, the paper is organized as follows. In Section 1, we will recall some basic conceptions on Alexandrov spaces.
In Sections 2 and 3, we will give a proof of Theorem A for curvature `$\geq\ (\leq)\ k$' and `$=k$' respectively.
In Section 4, we shall give proofs for Theorem C and Corollary D.


\section{On Alexandrov spaces}

In this section, we will recall the definition and some basic properties of Alexandrov spaces, which
will be used in the proof of Theorem A.

First of all, it is well known that Theorem \ref{thm0.1} has the following twin version.

\begin{theorem}[\cite{CE}] \label{thm1.1}
Let $M$ be a complete Riemannian manifold. Then
$\sec_M\geq k\ (\leq k)$  if and only if for any $x\in M$ there exists a neighborhood $U_x$ of $x$ such that

\vskip1mm

\noindent{\rm (1.1)}\ \  for any $q\in U_x$, $[pr]\subset U_x$, and $\tilde q\in\mathbb S^2_k$ and $[\tilde p\tilde r]\subset \mathbb S^2_k$ with $|\tilde q\tilde p|=|qp|$, $|\tilde q\tilde r|=|qr|$ and $|\tilde p\tilde r|=|pr|$, we have that, for any $s\in [pr]$ and $\tilde s\in [\tilde p\tilde r]$ with $|ps|=|\tilde p\tilde s|$,
$$|qs|\geq\ (\leq) \ |\tilde q\tilde s|.$$
\end{theorem}

We now, based on Theorem \ref{thm1.1}, can give the definition of the Alexandrov space in Theorem A.

\begin{defn}[cf. \cite{BGP}, \cite{AKP}] \label{dfn1.2} {\rm
A locally compact length space $X$ is called an {\it Alexandrov space}
if for any $x\in X$ there is a real number $k_x$ and a neighborhood $U_x$ of $x$
such that the corresponding condition (1.1) with respect to $\mathbb S^2_{k_x}$ holds, and
$X$ is said to be {\it of curvature $\text{\rm cur}_X\geq k_x$ or $\leq k_x$ on $U_x$} according to `$|qs|\geq|\tilde q\tilde s|$' or `$|qs|\leq|\tilde q\tilde s|$' respectively.}
\end{defn}

Given an Alexandrov space $X$, we call $x\in X$  a {\it CBB-type} (resp. {\it CBA-type}) {\it point} when $\text{\rm cur}_X\geq k_x$ (resp. $\leq k_x$) around $x$; and we call $X$ a {\it CBB-type} (resp. {\it CBA-type}) {\it Alexandrov space} if each $x\in X$ is a CBB-type (resp. CBA-type) point.

It is obvious that Alexandrov spaces include Riemannian manifolds. To a general Alexandrov space $X$,
a significant difference from a Riemannian manifold is that a geodesic (locally shortest path) on $X$ might not be
prolonged even when $X$ is complete (in the sense of distance topology).

In an Alexandrov space $X$, we can define an angle $\angle yxz$ between two minimal geodesics $[xy]$ and $[xz]$.
Assume that $X$ is of cur$_X\geq k_x$ around $x$. Let $a\in[xy]$ and $b\in [xz]$.
By condition (1.1), $\tilde\angle_{k_x} axb$ is non-decreasing when $a,b$ converge to $x$ (\cite{BGP}),
i.e. $\lim\limits_{a,b\to x} \tilde\angle_{k_x} axb$ exists. And note that, for any $k\neq k_x$,
$\lim\limits_{a,b\to x}\tilde\angle_{k} axb=\lim\limits_{a,b\to x}\tilde\angle_{k_x}axb.$
So, we can define
\stepcounter{equation}
\begin{equation}\label{eqn1.2}
 \angle yxz\triangleq\lim_{a,b\to x} \tilde\angle_{k_x} axb.
\end{equation}
Similarly, we can also define $\angle yxz $ if $X$ is of cur$_X\leq k_x$ around $x$
because, in such a situation, $\tilde\angle_{k_x} axb$ is non-increasing when $a,b$ converge to $x$ (\cite{AKP}).

By (\ref{eqn1.2}) and Definition \ref{dfn1.2}, it is not hard to see that Theorem \ref{thm0.1} also holds for $X$ (\cite{BGP},\cite{AKP}):

\begin{theorem} \label{thm1.3}
Let $X$ be a complete CBB-type Alexandrov space. Then $X$ is of $\text{cur}_X\geq\ (\leq)\ k$ if and only if
for any $x\in X$ there exists a neighborhood $U_x$ of $x$ such that the inequality in Theorem \ref{thm0.1} holds for any $[pq],[pr]\subset U_x$,
or equivalently, for any triangle $\triangle pqr\subset U_x$ (i.e. a union of three minimal geodesics $[pq],[qr],[pr]$) we have that
$\angle qpr\geq\ (\leq) \ \tilde\angle_{k} qpr$, $\angle pqr\geq\ (\leq) \ \tilde\angle_{k} pqr$ and $\angle prq\geq\ (\leq) \ \tilde\angle_{k} prq$.
\end{theorem}

\noindent{\bf Theorem 1.3$'$.} {\it
Let $X$ be a complete CBA-type Alexandrov space. Then $X$ is of $\text{cur}_X\leq k$ (resp. $\geq k$) if and only if (resp. only if)
the condition for $\text{cur}_X\leq k$ (resp. $\geq k$) in Theorem \ref{thm1.3} holds.}

\vskip2mm

In Theorem 1.3$'$, the ``if'' (i.e. the sufficiency of the condition) for $\text{cur}_X\geq k$ needs an additional condition that
$\angle qsp+\angle qsr=\pi$ for all $s\in [pr]^\circ$ ([BGP], [AKP]).

And similar to Riemannian case, we have the first variation formula on Alexandrov spaces.

\begin{lemma}[\cite{BGP},\cite{AKP}]\label{lem1.4}
 	Let $X$ be a complete Alexandrov space.
 	Then for any $x\in X$, there is a neighborhood $U_x$ of $x$ such that, for any $[pq],[pr]\subset U_x$ and $p_i\in [pr]$ with $p_i\to p$ as $i\to\infty$,
 	\begin{equation}\label{eqn1.4}
 	\begin{cases}
 	|qp_i|=|qp|-|pp_i| \cdot\cos\angle qpr +o(|pp_i|), & \text{ $x$ is a CBA-type point;}  \\
 	|qp_i|=|qp|-|pp_i|\cdot \cos|\uparrow_p^r\Uparrow_p^q| +o(|pp_i|), & \text{ $x$ is a CBB-type point.} \\
 	\end{cases}\end{equation}
\end{lemma}

In this paper, for a given $[xy]$, $\uparrow_x^y$ denotes its direction at $x$ (in Rimannian case, $\uparrow_x^y$ is just
the unit tangent vector of $[xy]$ from $x$ to $y$); and $\Uparrow_x^y$ denotes
the union of directions of all minimal geodesics from $x$ to $y$.
Note that by Definition \ref{dfn1.2}, in the case where $X$ has an upper curvature bound on $U$,
there is a unique minimal geodesic between any two distinct points in $U$.

\vskip2mm

\begin{remark}{\rm
For $\text{cur}_X\geq k$, Theorem \ref{thm1.3} (and \ref{thm0.1}) guarantees a global version of itself {\rm(\cite{BGP}, cf. \cite{Pl} and \cite{Wa})}, which is the well-known Toponogov's Theorem;
namely, $\text{cur}_X\geq k$ implies that (1.1) holds for any $q\in X$ and $[pr]\subset X$. However, there is no global version for $\text{cur}_X\leq k$ in general (\cite{AKP}).
As a result, (\ref{eqn1.4}) has a global version (the first variation formula) on a complete Alexandrov space $X$ with $\text{cur}_X\geq k$, but does not on $X$ with only an upper curvature bound.}
\end{remark}

Moreover, the angles on Alexandrov spaces satisfy the following property.

\begin{lemma}[\cite{BGP}, \cite{AKP})]\label{lem1.6}
Let $X$ be a complete Alexandrov space.
Then for any $x\in X$, there is a neighborhood $U_x$ of $x$ such that, for any
$[qp],[r_1r_2]\subset U_x$ with $p\in[r_1r_2]^\circ$,
\begin{equation}\label{eqn1.5}
\begin{cases}
\angle qpr_1+\angle qpr_2\geq\pi, & \text{  $x$ is a CBA-type point;} \\
\angle qpr_1+\angle qpr_2=\pi, & \text{ $x$ is a CBB-type point.} \\
\end{cases}
\end{equation}
As a result, if $x$ is a CBB-type point and if in addition $|qp|=|q[r_1r_2]|$, then $\angle qpr_1=\angle qpr_2=\frac\pi2.$
\end{lemma}

And it is easy to see that the angles on Alexandrov spaces have semi-continuity. Namely, given a complete Alexandrov space $X$
which has a lower (resp. upper) curvature bound on a neighborhood $U$, if $[p_iq_i]\to [pq]$ and $[p_ir_i]\to [pr]$ as $i\to\infty$ in $U$,
then $\angle qpr\leq\liminf\limits_{i\to\infty}\angle q_ip_ir_i$ (resp. $\angle qpr\geq\limsup\limits_{i\to\infty}\angle q_ip_ir_i$).
This, together with (\ref{eqn1.5}), implies the following continuity.

\begin{lemma}\label{lem1.7}
Let $X$ be a complete Alexandrov space, and let $U$ be a neighborhood in $X$.
Suppose that $[qp],[q_ip_i], [r_1r_2]\subset U$ with $p,p_i\in[r_1r_2]^\circ$ and $[q_ip_i]\to[qp]$ as $i\to\infty$.
If $X$ has a lower curvature bound on $U$,
or if $X$ has an upper curvature bound on $U$ and $\angle qpr_1+\angle qpr_2=\pi$, then
$$\lim_{i\to \infty}\angle q_ip_ir_1=\angle qpr_1 \text{ and } \lim_{i\to \infty}\angle q_ip_ir_2=\angle qpr_2.$$
\end{lemma}

We now end this section with some conceptions which only apply to CBB-type Alexandrov spaces (\cite{BGP}).
A CBB-type Alexandrov space $X$ has the conception of dimension, and
the space of directions $\Sigma_pX$ at any $p\in X$ is an Alexandrov space with curvature $\geq 1$
and has a dimension one less than $X$. If $\Sigma_pX$ is isometric to a unit sphere, we say that $p$ is a Riemannian point.

As a result,
by induction we can define $p$ to be a boundary or interior point if $\Sigma_pX$ contains boundary or no boundary point respectively. Usually, we denote by $X^\circ$ and $\partial X$ the set of interior points  and boundary points respectively. Note that $\partial X$ may be not empty even if $X$ is complete.

As another result, for any $p\in X$, we can define the tangent cone $C_pX$, which is a metric cone over $\Sigma_pX$.
$C_pX$ plays an important role in studying CBB-type Alexandrov spaces because of:

\begin{lemma}[\cite{BGP}]\label{lem1.8}
	Let $X$ be a complete CBB-type Alexandrov space.
	Then with base point $p\in X$, $(\lambda X, p)$  converge in the Gromov-Hausdorff sense to $C_pX$ as $\lambda\to+\infty$.
\end{lemma}

In this paper, $\lambda X$ denotes $X$ endowed with the metric $\lambda\cdot d$, where $d$ is the original metric on $X$.


\section{Proof of Theorem A for curvature $\geq \ (\leq)\ k$ }

In this section, we will show the former part of Theorem A,
i.e. the sufficiency and necessity of the condition for curvature $\geq k$ or $\leq k$ in Theorem A.
By Theorems \ref{thm1.3} and 1.3$'$, it suffices to verify the sufficiency, and the verification shall be proceeded
according to:

Case 1: For curvature $\geq k$ around a CBB-type point $x\in X$;

Case 2: For curvature $\geq k$ around a CBA-type point $x\in X$;

Case 3: For curvature $\leq k$ around a CBA-type point $x\in X$;

Case 4: For curvature $\leq k$ around a CBB-type point $x\in X$.

\vskip2mm

For the convenience of readers, we first give a rough idea of our proof.
For instance, in Case 1, if the curvature of $X$ is not $\geq k$ around the CBB-type point $x$,
then by Theorem \ref{thm1.3} there must be a triangle $\triangle pqr$ containing a `bad' angle, say $\angle qpr$,  i.e. $\angle qpr<\tilde\angle_k qpr$.
A key observation is that such a triangle can be cut into two (smaller) triangles
and at least one of them still contains a `bad' angle (cf. \cite{Wa}), which is guaranteed by the lemma right below.
By repeating such a cutting operation finite times, we will get a triangle which contradicts the condition for curvature $\geq k$ in Theorem A.

In this paper, for a given $\triangle qpr$, its comparison triangle is defined to be a
$\triangle\tilde p\tilde q\tilde r\subset\mathbb S^2_k$ with
$|\tilde p\tilde q|=|pq|$, $|\tilde p\tilde r|=|pr|$ and $|\tilde q\tilde r|=|qr|$.

\begin{lemma}[cf. \cite{Wa}]\label{lem2.1}
Let $x$ be a {\text CBB}-point in  a complete Alexandrov space,
and let $k$ be a real number. Then there is a neighborhood $U_x$ of $x$ such that,
for any $\triangle qr_1r_2\subset U_x$ and its comparison triangle $\triangle\tilde q\tilde r_1\tilde r_2\subset\mathbb S^2_k$,
if $|qs|-|\tilde q\tilde s|$ with $s\in[r_1r_2], \tilde s\in[\tilde r_1\tilde r_2]$ and $|r_is|=|\tilde r_i\tilde s|$
attains a negative minimum at $s_0\in [r_1r_2]^\circ$, then any $[qs_0]$ satisfies
\begin{equation}\label{eqn2.1}
\angle qs_0r_1<\tilde\angle_k qs_0r_1 \ \text{ or }\ \angle qs_0r_2<\tilde\angle_k qs_0r_2;
\end{equation}
in particular,
\begin{equation}\label{eqn2.5}
\text{if $|r_is_0|\ll |r_iq|$ for $i=1$ or $2$, then $\angle qs_0r_i<\tilde \angle_k qs_0r_i$.}
\end{equation}
\end{lemma}

Note that if $k$ is positive in the lemma, then the larger $k$ is,
the smaller $U_x$ should be to guarantee that $\triangle qr_1r_2$ has a comparison triangle in $\mathbb S^2_k$.

\begin{proof}
Since $x$ is a {\text CBB}-point, by Lemmas \ref{lem1.4} and \ref{lem1.6}, there is a neighborhood $U_x$ of $x$ such that
we can apply (\ref{eqn1.4}) and (\ref{eqn1.5}) on it.
For any $[qs_0]$, by (\ref{eqn1.4}), the negative minimum of $|qs|-|\tilde q\tilde s|$ at $s_0$ implies
$$
\angle qs_0r_i\geq\angle \tilde q\tilde s_0\tilde r_i,\ \ i=1,2.
$$
It then has to hold that
\begin{equation}\label{eqn2.3}
\angle qs_0r_i=\angle \tilde q\tilde s_0\tilde r_i,\ \ i=1,2	\end{equation}
because, by (\ref{eqn1.5}), we have that
\begin{equation}\label{eqn2.4}	
\angle qs_0r_1+\angle qs_0r_2=\pi.
\end{equation}
On the other hand, since $|qs_0|<|\tilde q\tilde s_0|$, there is $\tilde q'\in[\tilde q\tilde s_0]^\circ$
such that $|\tilde q'\tilde s_0|=|qs_0|$. It is clear that at least one of
`$|\tilde q'\tilde r_i|<|\tilde q\tilde r_i|$' holds.
It follows that at least one of
`$\angle \tilde q\tilde s_0\tilde r_i<\tilde \angle_k q s_0r_i$'
holds; especially, if $|r_is_0|\ll |r_iq|$ for $i=1$ or $2$, then $\angle \tilde q\tilde s_0\tilde r_i<\tilde \angle_k qs_0r_i$.
As a result, the lemma follows.
\end{proof}

\begin{remark}\label{rem2.2}{\rm
(2.2.1) In Lemma \ref{lem2.1}, if $|qs|-|\tilde q\tilde s|$ attains a positive maximum at $s_0\in [r_1r_2]^\circ$, then
	 	similar to (\ref{eqn2.1}) either $\angle qs_0r_2>\tilde\angle_k qs_0r_2$
	 	for the $[qs_0]$ with $|\uparrow_{s_0}^q\uparrow_{s_0}^{r_1}|=|\Uparrow_{s_0}^q\uparrow_{s_0}^{r_1}|$,
	 	or $\angle qs_0r_1>\tilde\angle_k qs_0r_1$
	 	for the $[qs_0]$ with $|\uparrow_{s_0}^q\uparrow_{s_0}^{r_2}|=|\Uparrow_{s_0}^q\uparrow_{s_0}^{r_2}|$.
	 	And similar to (\ref{eqn2.5}), if $\angle \tilde r_i\tilde q\tilde s_0<\frac\pi2$ for $i=1$ or $2$, then
	 	$\angle \tilde q\tilde s_0\tilde r_i>\tilde \angle_k qs_0r_i$ and thus the corresponding $\angle qs_0r_i>\tilde\angle_k qs_0r_i$.

\vskip1mm
	 	
\noindent(2.2.2) In Lemma \ref{lem2.1}, if $x$ is a {\text CBA}-point, the lemma is not true
	 	unless $\angle qs_0r_1+\angle qs_0r_2=\pi$ (see (\ref{eqn2.4}); note that it may occur that $\angle qs_0r_1+\angle qs_0r_2>\pi$ by (\ref{eqn1.5})).
	
\vskip1mm
 	
\noindent (2.2.3) In Lemma \ref{lem2.1}, if $x$ is a {\text CBA}-point, and if
	 	$|qs|-|\tilde q\tilde s|$ attains a positive maximum at $s_0\in [r_1r_2]^\circ$, then similar to (\ref{eqn2.1})
	 	we have that $\angle qs_0r_i>\tilde\angle_k qs_0r_i$ for $i=1$ or $2$.
	 	(Here, there is a unique minimal geodesic between $q$ and $s_0$ because of the CBA-property.)
	 	And similar to (\ref{eqn2.5}), if $\angle \tilde r_i\tilde q\tilde s_0<\frac\pi2$ for $i=1$ or $2$, then
	 	$\angle \tilde q\tilde s_0\tilde r_i>\tilde \angle_k qs_0r_i$ and thus  $\angle qs_0r_i>\tilde\angle_k qs_0r_i$.}
\end{remark}

Now, according to the four cases listed in the beginning of this section, we begin to prove the former part of Theorem A case by case.

\begin{proof}[Proof for Case 1]\

In this case, around a CBB-type point $x\in X$, we will prove:

\newcounter{eqn1}
\noindent \stepcounter{equation} \setcounter{eqn1}{\value{equation}} (\arabic{section}.\arabic{equation}\label{eqn2.5.1})\ \
{\it $X$ is of curvature $\geq k$ if there is a neighborhood $U_x$ of $x$
	such that,
	\begin{equation}\label{eqn2.6}
	\text{ for any $q\in U_x$ and $[r_1r_2]\subset U_x$, if $p\in [r_1r_2]^\circ$
		satisfies $|qp|=|q[r_1r_2]|$, then
		$\tilde\angle_kqpr_i\leq \frac\pi2$.}
	\end{equation}  }
\hskip5mm If $X$ is not of curvature $\geq k$ around $x$, then we claim that
there is a triangle $\triangle qpr\subset U_x$ such that
\begin{equation}\label{eqn2.7}
	\text{$|qs|=|q[pr]|$ for some $s\in [pr]^\circ$, and $|qs|<|\tilde q\tilde s|$ for $\tilde s\in[\tilde p\tilde r]$  with $|\tilde s\tilde p|=|sp|$,}
\end{equation}
where $[\tilde p\tilde r]$ belongs to the comparison triangle $\triangle \tilde q\tilde p\tilde r\subset\mathbb S^2_k$ of $\triangle qpr$.
Nevertheless, the $\triangle qpr$ contradicts (\ref{eqn2.6}). In fact, note that either $\angle\tilde q\tilde s\tilde p\geq \frac\pi2$
or $\angle\tilde q\tilde s\tilde r\geq \frac\pi2$, say $\angle\tilde q\tilde s\tilde p\geq \frac\pi2$.
Then by applying (\ref{eqn2.6}) on $\triangle qpr$, it has to hold that $|qs|\geq|\tilde q\tilde s|$, a contradiction.

We now need only  to verify the claim, i.e. to show the existence of the desired triangle. By Theorem \ref{thm1.3},
if $X$ is not of curvature $\geq k$ around $x$,
there exists a sufficiently small triangle $\triangle pqr\subset U_x$
which contains a `bad' angle, say $\angle qpr$, i.e.
$\angle qpr<\tilde\angle_k qpr.$

First of all, observe that the badness of $\angle qpr$ implies that there is $\bar s\in [pr]\setminus\{p\}$ and $\bar t\in [pq]\setminus\{p\}$ such that,
for all $s\in[p\bar s]^\circ$ and $t\in [p\bar t]^\circ $,
\begin{equation}\label{eqn2.8}
|\bar ts|<|\tilde{\bar t}\tilde s|,  |\bar st|<|\tilde{\bar s}\tilde t|, |\bar s\bar t|=|\tilde{\bar s}\tilde{\bar t}|  \text{ and } \angle \bar tp\bar s-\tilde\angle_k\bar tp\bar s=\angle qpr-\tilde\angle_k qpr,
\end{equation}
where $\tilde{\bar s},\tilde s,\tilde{\bar t},\tilde t$ belong to the comparison triangle
$\triangle\tilde p\tilde q\tilde r\subset\Bbb S^2_k$ of $\triangle qpr$ and correspond to $\bar s,s,\bar t,t$ respectively with
$\tilde s\in [\tilde p\tilde r]$ and $|\tilde s\tilde p|=|sp|$, etc.
In fact, by the badness of $\angle qpr$ and Lemma \ref{lem1.4},
there is $s_{1}\in [pr]\setminus\{p\}$ such that
\begin{equation}\label{C2.4}
|qs_{1}|=|\tilde{q}\tilde s_{1}| \text{ and } |qs|<|\tilde{q}\tilde s|
\end{equation}
for all $s\in[ps_{1}]^\circ$ and $\tilde s_{1},\tilde s\in [\tilde{p}\tilde{r}]^\circ$ with $|\tilde s_{1}\tilde{p}|=|s_{1}p|$
and $|\tilde s\tilde{p}|=|sp|$. Note that $\triangle\tilde{q}\tilde{p}\tilde{s}_{1}$  is a comparison triangle of $\triangle qps_{1}$ with
$$\angle qps_{1}-\tilde\angle_k  qps_{1}=\angle qpr-\tilde\angle_k  qpr,$$
so $\angle qps_{1}$ is still `bad' in $\triangle qps_{1}$. And similarly, there is $t_{1}\in [pq]\setminus\{p\}$ such that
\begin{equation}\label{C2.5}
|s_{1}t_{1}|=|\tilde s_{1}\tilde t_{1}| \text{ and } |s_{1}t|<|\tilde s_1\tilde t|
\end{equation}
for all $t\in[pt_{1}]^\circ$ and $\tilde t_{1},\tilde t\in[\tilde{p}\tilde{q}]$ with $|\tilde t_{1}\tilde{ p}|=|t_{1}p|$ and $|\tilde t\tilde{p}|=|tp|$.
Furthermore, we can locate an $s_{2}\in [ps_{1}]\setminus\{p\}$ similar to $s_{1}$ and a
$t_{2}\in [pt_{1}]\setminus\{p\}$ similar to $t_{1}$. Then we can get a sequence of $\{s_{j}\}_{j=1}^\infty$ and $\{t_{j}\}_{j=1}^\infty$
(it may occur that $s_{j}=s_{j_0}$ and $t_{j}=t_{j_0}$ for all $j\geq j_0$) such that
\begin{equation}\label{C2.6}
\angle s_{j}pt_{j-1}-\tilde\angle_ks_{j}pt_{j-1}=
\angle s_{j}pt_{j}-\tilde\angle_k s_{j}pt_{j}=\angle qpr-\tilde\angle_k  qpr,
\end{equation}
which, by (\ref{eqn1.2}), implies that
$$s_{j}\to \bar s\in [pr]\setminus\{p\} \text{ and } t_{j}\to \bar t\in [pq]\setminus\{p\}
\text{ as } j\to\infty.$$
By the corresponding (\ref{C2.4}) and (\ref{C2.5}) for each $j$ and (\ref{C2.6}), we can conclude that
$\bar s$ and $\bar t$ must satisfy the two equalities in (\ref{eqn2.8}). Then up to repeating this process on $\triangle \bar tp\bar s$,
we can assume that $\bar s$ and $\bar t$ also satisfy the two inequalities in (\ref{eqn2.8}). Namely, we have found the desired $\bar s$ and $\bar t$.

Note that (\ref{eqn2.8}) implies that
\begin{equation}\label{C2.7}
\tilde\angle_k \bar t\bar sp\geq |\uparrow_{\bar s}^p\Uparrow_{\bar s}^{\bar t}| \text{ and }
\tilde\angle_k \bar s\bar tp\geq |\uparrow_{\bar t}^p\Uparrow_{\bar t}^{\bar s}|.
\end{equation}
On the other hand, note that $\triangle pqr$ is sufficiently small, so is any $\triangle\bar t p\bar s$;
and thus at least one of $\tilde\angle_k \bar t\bar sp$ and $\tilde\angle_k \bar s\bar tp$ is an acute angle.
This together with (\ref{C2.7}) implies that there is a triangle $\triangle \bar tp\bar s$ such that
\begin{equation}\label{C2.8}
\text{at least one of $\angle \bar t\bar sp$ and $\angle \bar s\bar tp$ is acute, say $\angle \bar t\bar sp$.}
\end{equation}
In addition, if $\angle qpr<\frac\pi2$, then by Lemma \ref{lem1.4} there is $s\in [p\bar s]^\circ$ such that
$|\bar ts|=|\bar t[p\bar s]|$, i.e. $\triangle \bar tp\bar s$ is our desired triangle.

Hence, we need only to show that it can be assumed that $\angle qpr<\frac\pi2$.
Since $\angle qpr$ is a `bad' angle,
for $s\in[pr]^\circ$ close to $p$ and $\tilde s\in[\tilde p\tilde r]^\circ$ with $|\tilde s\tilde p|=|sp|$
we have that $|qs|<|\tilde q\tilde s|$ (by Lemma \ref{lem1.4}). Then there is $r'\in [pr]^\circ$
and $\tilde r'\in[\tilde p\tilde r]^\circ$ with $|pr'|=|\tilde p\tilde r'|$ such that
\begin{equation}\label{eqn2.9}
|qr'|-|\tilde q\tilde r'|=\min\limits_{s\in[pr],
\tilde s\in[\tilde p\tilde r], |ps|=|\tilde p\tilde s|}\{|qs|-|\tilde q\tilde s|\}<0,
\end{equation}
so via Lemma \ref{lem2.1} we can conclude that
\begin{equation}\label{eqn2.10}
\text{ $\angle qr'p$ is a `bad' angle in $\triangle qr'p$, or $\angle qr'r$ is a `bad' angle in $\triangle qr'r$}.
\end{equation}
And in this situation, (\ref{eqn2.3}) means that
\begin{equation}\label{eqn2.11}
\angle qr'p=\angle\tilde q\tilde r'\tilde p \text{ and } \angle qr'r=\angle\tilde q\tilde r'\tilde r.
\end{equation}
Note that if $\angle qpr\geq\frac\pi2$, then it follows from the badness of $\angle qpr$ that $\angle \tilde p\tilde q\tilde r\ (=\tilde\angle_k qpr)>\frac\pi2$.
And note that $\triangle \tilde p\tilde q\tilde r$ is sufficiently small in $\mathbb S^2_k$, so it is easy to see that
\begin{equation}\label{eqn2.12}
\angle \tilde q\tilde r'\tilde p<\frac\pi2 \ \text{ and }\ \angle \tilde q\tilde r'\tilde r>\frac\pi2.
\end{equation}
By (\ref{eqn2.10})-(\ref{eqn2.12}), $\triangle qr'p$ contains an acute `bad' angle if the angle $\angle qr'p$ is `bad'; otherwise,
we can repeat such a cutting oparation on $\triangle qr'r$. For convenience, we also let $\triangle qpr$ denote
the $\triangle qr'r$. Note that up to repeating such a cutting oparation finite times,
we can assume that $|pr|\ll |pq|$, so plus (\ref{eqn2.5}) we can conclude that $\angle qr'p$ is an acute `bad' angle in $\triangle qr'p$.
This means that we can assume that $\angle qpr<\frac\pi2$.
\end{proof}

\begin{proof}[Proof for Case 2]\
	
In this case, we shall prove (\ref{eqn2.5.1}.\arabic{eqn1}) around a CBA-type point $x\in X$.
Compared with the proof for Case 1, the only difference and difficulty here
is why (\ref{eqn2.10}) holds. Note that Lemma \ref{lem2.1} fails to work here (see (2.2.2) in Remark \ref{rem2.2}).
Namely, the proof for Case 2 will be done if one can show that: {\it for a CBA-type point $x$,
if the $U_x$ in Lemma \ref{lem2.1} satisfies (\ref{eqn2.6}) additionally,
then the conclusion of Lemma \ref{lem2.1} still holds}.
By (2.2.2), it suffices to show that $\angle qs_0r_1+\angle qs_0r_2=\pi$,
i.e., the possible case `$\angle qs_0r_1+\angle qs_0r_2>\pi$' does not occur at all here.

Let $q'\in [qs_0]$ be sufficiently close to $s_0$, and let $s'\in [r_1r_2]$ such that $|q's'|=|q'[r_1r_2]|$
(for $[qs_0]$ and $[r_1r_2]$ refer to the proof of Lemma \ref{lem2.1}).
Note that $s'$ is close to $s_0$, so we can assume that $s'$ lies in $[r_1r_2]^\circ$.
Then for $\tilde s'\in[\tilde r_1\tilde r_2]$ ($\subset\triangle \tilde q\tilde r_1\tilde r_2\subset \mathbb S^2_k$)
with $|\tilde r_1\tilde s'|=|r_1s'|$, there is $\tilde q'\in\mathbb S_k^2$ such that
\begin{equation}\label{eqn2.17}
|\tilde r_1\tilde q'|+|\tilde q'\tilde r_2|=|r_1q'|+|q'r_2|
\end{equation}
and
\begin{equation}\label{eqn2.18}
\angle\tilde q'\tilde s'\tilde r_1=\angle\tilde q'\tilde s'\tilde r_2=\frac\pi2.
\end{equation}
By Lemma \ref{lem1.4}, it is easy to see that
\begin{equation}\label{eqn2.19}
|r_1q'|+|q'r_2|=|r_1r_2|-|q's_0|\cdot (\cos\angle qs_0r_1+\cos\angle qs_0r_2)+o(|q's_0|),
\end{equation}
and that
\begin{equation}\label{eqn2.20}
|\tilde r_1\tilde q'|+|\tilde q'\tilde r_2|=|\tilde r_1\tilde r_2|-|\tilde q'\tilde s'|\cdot (\cos\angle \tilde q'\tilde s'\tilde r_1+\cos\angle \tilde q'\tilde s'\tilde r_2)+o(|\tilde q'\tilde s'|)=|r_1r_2|+o(|\tilde q'\tilde s'|).
\end{equation}
Hence, if $\angle qs_0r_1+\angle qs_0r_2>\pi$, then from (\ref{eqn2.17}), (\ref{eqn2.19}) and (\ref{eqn2.20}) we can see that
\begin{equation}\label{eqn2.21}
|q's'|\leq|q's_0|<|\tilde q'\tilde s'|.
\end{equation}
However, putting `$|q's'|=|q'[r_1r_2]|$', `$|\tilde r_i\tilde s'|=|r_is'|$', (\ref{eqn2.18}) and (\ref{eqn2.21}) together,
we can apply (\ref{eqn2.6}) to conclude that
$|r_1q'|+|q'r_2|<|\tilde r_1\tilde q'|+|\tilde q'\tilde r_2|,$
which contradicts (\ref{eqn2.17}).
\end{proof}

\begin{proof}[Proof for Case 3]\

In this case, around a CBA-type point $x\in X$, we will prove:

\vskip1mm

\newcounter{C2.16}
\noindent \stepcounter{equation} \setcounter{C2.16}{\value{equation}} (\arabic{section}.\arabic{equation})
{\it $X$ is of curvature $\leq k$ if there is a neighborhood $U_x$ of $x$
such that,
\begin{equation}\label{eqn2.25}
\text{for any $q\in U_x$ and $[r_1r_2]\subset U_x$, if $p\in [r_1r_2]^\circ$
satisfies $|qp|=|q[r_1r_2]|$, then
$\tilde\angle_kqpr_i\geq \frac\pi2$.}
\end{equation}
}
The proof is almost a copy of that for Case 1 with reversing the directions of
the corresponding inequalities. (Hint: ``$\angle qpr$ is a `bad' angle'' here means that
$\angle qpr>\tilde\angle_k qpr$. Then by Lemma \ref{lem1.4} there is $r'\in[pr]^\circ$ and
$\tilde r'\in[\tilde p\tilde r]^\circ$, which correspond to $r'$ and $\tilde r'$ satisfying (\ref{eqn2.9}), such that
\begin{equation}\label{eqn2.26}
|qr'|-|\tilde q\tilde r'|=\max\limits_{s\in[pr],
	\tilde s\in[\tilde p\tilde r], |ps|=|\tilde p\tilde s|}\{|qs|-|\tilde q\tilde s|\}>0.
\end{equation}
Furthermore, we can apply (2.2.3) to see (\ref{eqn2.10}).)
So, we only point out the main two differences here.

One is how to see (\ref{C2.8}). In Case 1, (\ref{C2.7}) is a key, but here
the corresponding (\ref{C2.7}) has inverse directions.
However, the CBA-property of $x$ here implies (\ref{C2.8}) directly when $U_x$ is small enough.

The other is how to see the acuteness of $\angle qr'p$ as in the end of
the proof for Case 1. In Case 1, a key is (\ref{eqn2.12}) which is due to `$\frac\pi2\leq\angle qpr<\tilde\angle_k qpr$';
but here the `bad' of $\angle qpr$ means that $\angle qpr>\tilde\angle_k qpr$.
However, the CBA-property of $x$ here with `$\frac\pi2\leq\angle qpr$' implies the acuteness of $\angle qr'p$
directly as long as $U_x$ is small enough.
\end{proof}

\begin{proof}[Proof for Case 4]\

In this case, we should prove (\ref{eqn2.5.1}.\arabic{C2.16}) around a CBB-type point $x\in X$.
If it is not true, then similarly, under the assumption that $X$ is not of curvature $\leq k$ around $x$,
in $U_x$ (in (\ref{eqn2.5.1}.\arabic{C2.16}))
we just need to locate a triangle contradicting (\ref{eqn2.25}), i.e.
a triangle satisfying the lemma right below.	
\end{proof}

\begin{lemma}\label{lem2.3}
Let $x\in X$ be a CBB-type point. If $X$ is not of curvature $\leq k$ around $x$,
and if a sufficiently small neighborhood $U_x$  of $x$ satisfies (\ref{eqn2.25}), then
there is a triangle $\triangle qpr\subset U_x$ such that $|qs|=|q[pr]|$ and $|qs|>|\tilde q\tilde s|$ for some
$s\in[pr]^\circ$ and $\tilde s\in[\tilde p\tilde r]^\circ$ with $|\tilde s\tilde p|=|sp|$, where
$[\tilde p\tilde r]$ belongs to the comparison triangle $\triangle \tilde q\tilde p\tilde r\subset\mathbb S^2_k$ of $\triangle pqr$.
\end{lemma}

Actually, the proofs for Cases 1-3 mainly show a corresponding Lemma \ref{lem2.3}.
Compared with them, a main difficulty here is that we can not conclude
(\ref{C2.8}) because we have no inequalities in (\ref{C2.7}) as in Cases 1-2 nor the CBA-property in Case 3.
Another main difficulty here appears in looking for a triangle with an acute `bad' angle.
In Cases 1-3, a `bad' angle leads to a situation
where we can apply Lemma \ref{lem2.1} or (2.2.3) to locate a smaller triangle with a `bad' angle.
And step by step, we can locate the desired triangle.
However, in Case 4, such a method fails when we try to apply (2.2.1) unless
there is a unique minimal geodesic between any two points in $U_x$.

To overcome the second difficulty right above,
we have the following key observation from (\ref{eqn2.25}).

\begin{lemma}\label{lem2.4}
Let $x\in X$ be a CBB-type point, and let $U_x$ be a sufficiently small neighborhood of $x$ satisfying (\ref{eqn2.25}).
Then there is a unique minimal geodesic between any two distinct points in $U_x$.
\end{lemma}

In the proof of Lemma \ref{lem2.4}, we need the following property of CBB-type Alexandrov spaces.

\begin{lemma}[\cite{BGP}, \cite{LR}]\label{lem2.5}
Let $x\in X$ be a CBB-type point. Then there is a sufficiently small neighborhood $U$ of $x$ such that
if there are two minimal geodesics between two points $r_1$ and $r_2$ in $U$,
then they form an angle less than $\pi$ at $r_1$ or $r_2$.
\end{lemma}

\begin{proof}[Proof of Lemma \ref{lem2.4}]\

We argue by contradiction. Let $r_1$ and $r_2$ be two points in $U_x$, and assume that there are two minimal geodesics between them,
denoted by $[r_1r_2]_1$ and $[r_1r_2]_2$. By Lemma \ref{lem2.5},
$[r_1r_2]_1$ and $[r_1r_2]_2$ form an angle less than $\pi$ at  $r_1$ or $r_2$, say $r_1$, i.e. $|(\uparrow_{r_1}^{r_2})_1(\uparrow_{r_1}^{r_2})_2|<\pi$.
Then, by considering $\Sigma_{r_1}X$ (for it refer to Section 1), it is easy to see that there is a minimal geodesic $[r_1q]$
such that
\begin{equation}\label{eqn2.28}
|\uparrow_{r_1}^q(\uparrow_{r_1}^{r_2})_1|<|\uparrow_{r_1}^q(\uparrow_{r_1}^{r_2})_2|<\frac\pi2
\end{equation}
(in particular, if $|(\uparrow_{r_1}^{r_2})_1(\uparrow_{r_1}^{r_2})_2|<\frac\pi2$, we can let $[r_1q]=[r_1r_2]_1$).
We select $q_j\in [r_1q]\setminus\{r_1\}$ such that $q_j\to r_1$ as $j\to\infty$. Note that,
without loss of generality, we can assume that $|\uparrow_{r_1}^q(\uparrow_{r_1}^{r_2})_1|=|\uparrow_{r_1}^q\Uparrow_{r_1}^{r_2}|$.
By Lemma \ref{lem1.4}, it follows that, as $j\to\infty$,
\begin{equation}\label{eqn2.29}
|r_2q_j|=|r_1r_2|-|r_1q_j|\cos|\uparrow_{r_1}^q(\uparrow_{r_1}^{r_2})_1|+o(|r_1q_j|).
\end{equation}
On the other hand, by Lemma \ref{lem1.4}, `$|\uparrow_{r_1}^q(\uparrow_{r_1}^{r_2})_2|<\frac\pi2$' implies that there is $\bar q_j\in [r_1r_2]_2^\circ$
such that $|q_j\bar q_j|=|q_j[r_1r_2]_2|$.  And by Lemma \ref{lem1.8}, we have that, as $j\to\infty$,
\begin{equation}\label{eqn2.30}
|r_2\bar q_j|=|r_1r_2|-|r_1\bar q_j|=|r_1r_2|-|r_1q_j|\cos|\uparrow_{r_1}^q(\uparrow_{r_1}^{r_2})_2|+o(|r_1q_j|).
\end{equation}
It follows from (\ref{eqn2.28})-(\ref{eqn2.30}) that $|r_2q_j|<|r_2\bar q_j|$ for sufficiently large $j$.
Since $U_x$ can be sufficiently small, `$|r_2q_j|<|r_2\bar q_j|$' implies that $\tilde\angle_k q_j\bar q_jr_2<\frac\pi2$, which contradicts (\ref{eqn2.25}).
\end{proof}

In order to solve the first difficulty mentioned above, we will use the following technical property of CBB-type Alexandrov spaces,
especially (\ref{lem2.6}.2), in Step 4 of the proof of Lemma \ref{lem2.3} below.

\begin{lemma} \label{lem2.6}
Let $X$ be a complete Alexandrov space with curvature $\geq \kappa$, and let $p,q_i, r_i\in X$ with $q_i\to p$ and $r_i\to p$ as $i\to \infty$.  Then the following holds:

\vskip1mm

\noindent{\rm (\ref{lem2.6}.1) (\cite{BGP})} As $i\to\infty$, for any triangle $\triangle pq_ir_i$, we have that
$$\angle q_ipr_i-\tilde\angle_\kappa q_ipr_i\to 0,\ \angle pq_ir_i-\tilde\angle_\kappa pq_ir_i\to 0,
\ \angle pr_iq_i-\tilde\angle_\kappa pr_iq_i\to 0.$$

\vskip1mm

\noindent{\rm (\ref{lem2.6}.2)} Additionally, given $[pq_i]\ni p_i$ and $[q_ir_i]\ni s_i$,
if there is $c_1\in(0,1)$ and $c_2>0$ such that
\begin{equation}\label{eqn2.50}
|q_is_i|<c_1|q_ir_i| \text{ and }
\min\{\tilde\angle_\kappa q_ip_is_i,\tilde\angle_\kappa p_iq_is_i,\tilde\angle_\kappa p_is_iq_i\}>c_2
\end{equation}
for all $i$, then
as $i\to\infty$, for any triangle $\triangle p_iq_is_i$, we have that
$$
\angle q_ip_is_i-\tilde\angle_\kappa q_ip_is_i\to 0,\ \angle p_iq_is_i-\tilde\angle_\kappa p_iq_is_i\to 0,
\ \angle p_is_iq_i-\tilde\angle_\kappa p_is_iq_i\to 0.
$$
\end{lemma}

\begin{proof}

We just need to prove (\ref{lem2.6}.2).

First of all, by the reason for (\ref{eqn1.2}), we know that
$0\leq\angle p_iq_is_i-\tilde\angle_\kappa p_iq_is_i\leq \angle pq_ir_i-\tilde\angle_\kappa pq_ir_i$,
so it follows from (\ref{lem2.6}.1) that
\begin{equation}\label{eqn2.51}
\angle p_iq_is_i-\tilde\angle_\kappa p_iq_is_i\to 0 \text{ as } i\to\infty.
\end{equation}

Next, we show that $\angle q_is_i p_i-\tilde\angle_\kappa q_is_ip_i\to 0$ as $i\to\infty$.
Consider the point $u_i\in [q_ir_i]$ with $|u_iq_i|=2|s_iq_i|$ or $u_i=r_i$ when $|q_is_i|<\frac12|q_ir_i|$
or $|q_is_i|\geq\frac12|q_ir_i|$ respectively. Let $\triangle \tilde{q}_i\tilde{p}_i\tilde{u}_i\subset\mathbb S_\kappa^2$
be a comparison triangle of $\triangle q_ip_iu_i$, and let $\tilde s_i\in[\tilde{q}_i\tilde{u}_i]$
such that $|\tilde s_i\tilde q_i|=|s_iq_i|$. By Definition \ref{dfn1.2}, we know that $|p_is_i|\geq |\tilde p_i\tilde s_i|$.
On the other hand, by the same reason for (\ref{eqn2.51}), we have that $\angle p_iq_iu_i-\tilde\angle_\kappa p_iq_iu_i\to 0 \text{ as } i\to\infty$.
Plus (\ref{eqn2.50}) and by Theorem \ref{thm1.3}, we get that $|p_is_i|$ is almost equal to $|\tilde p_i\tilde s_i|$;
precisely,
\begin{equation}\label{eqn2.52}
\lim_{i\to\infty}\frac{|p_is_i|-|\tilde p_i\tilde s_i|}{|p_is_i|}=0.
\end{equation}
Together with (\ref{eqn2.50}), this implies that
\begin{equation}\label{eqn2.53}
\tilde\angle_\kappa q_is_ip_i-\angle\tilde{q}_i\tilde{s}_i\tilde{p}_i\to 0 \text{ and }
\tilde\angle_\kappa u_is_ip_i-\angle\tilde{u}_i\tilde{s}_i\tilde{p}_i\to 0 \text{ as } i\to\infty.
\end{equation}
Since $\angle q_is_ip_i\geq \tilde\angle_\kappa q_is_ip_i$ and  $\angle u_is_ip_i\geq \tilde\angle_\kappa u_is_ip_i$ (by Theorem \ref{thm1.3}),
and $\angle q_is_ip_i+\angle u_is_ip_i=\pi$ and $\angle\tilde{q}_i\tilde{s}_i\tilde{p}_i+\angle\tilde{u}_i\tilde{s}_i\tilde{p}_i=\pi$,
it follows from (\ref{eqn2.53}) that
$$\angle q_is_ip_i-\tilde\angle_\kappa q_is_ip_i\to 0 \text{ (and }
\angle u_is_ip_i-\tilde\angle_\kappa u_is_ip_i\to 0) \text{ as } i\to\infty.$$

At last, we show that $\angle q_ip_is_i-\tilde\angle_\kappa q_ip_is_i\to 0$ as $i\to\infty$.
Similarly, we consider $v_i\in [q_ip]$ with $|v_iq_i|=2|p_iq_i|$ or $v_i=p$ when $|q_ip_i|<\frac12|q_ip|$
or $|q_ip_i|\geq\frac12|q_ip|$ respectively, and a comparison triangle $\triangle \tilde{q}_i\tilde{s}_i\tilde{v}_i\subset\mathbb S_\kappa^2$
of $\triangle q_is_iv_i$ and $\tilde p_i\in [\tilde{q}_i\tilde{v}_i]$ with $|\tilde q_i\tilde p_i|=|q_ip_i|$.
It is easy to see that (\ref{eqn2.52}) still holds in the situation here.
And if we can show
\begin{equation}\label{eqn2.54}
\tilde\angle_\kappa s_ip_iq_i-\angle\tilde{s}_i\tilde{p}_i\tilde{q}_i\to 0 \text{ and }
\tilde\angle_\kappa s_ip_iv_i-\angle\tilde{s}_i\tilde{p}_i\tilde{v}_i\to 0 \text{ as } i\to\infty
\end{equation}
(similar to (\ref{eqn2.53})), then we can conclude that $\angle q_ip_is_i-\tilde\angle_\kappa q_ip_is_i\to 0$ as $i\to\infty$.
Indeed, we can similarly show (\ref{eqn2.54}) except possibly when $|q_ip_i|>\frac12|q_ip|$.
Note that $p_i$, unlike $s_i$ satisfying $|q_is_i|<c_1|q_ir_i|$,
may be sufficiently close to $p$ and even equal to $p$.
When $\frac12|q_ip|<|q_ip_i|\ (\leq|q_ip|)$, by the latter part of (\ref{eqn2.50})
there is $c_3>0$ such that $\tilde\angle_\kappa s_iq_ip>c_3$ and $\tilde\angle_\kappa q_ips_i>c_3$ for all $i$. Then plus
`$\angle q_ips_i-\tilde\angle_\kappa q_ips_i\to 0$ as $i\to\infty$ (by (2.6.1))' and (\ref{eqn2.52}),
we can apply the Law of Sine to conclude (\ref{eqn2.54}).
\end{proof}

We will end this section with proving Lemma \ref{lem2.3} (and so the proof for Case 4 is completed).

\begin{proof}[Proof of Lemma \ref{lem2.3}] \

First of all, in this proof, we can assume that there is a UNIQUE
minimal geodesic between any two distinct points in $U_x$ (see Lemma \ref{lem2.4}).

Since it is assumed that $X$ is not of curvature $\leq k$ around $x$, by Theorem \ref{thm1.3}
there is a  triangle $\triangle pqr\subset U_x$
containing a `bad' angle, say $\angle qpr$ (i.e. $\angle qpr>\tilde\angle_k qpr$).
Then similar to the existence of $\bar s$ and $\bar t$ in the proof for Case 1,
there is $\bar s\in [pr]\setminus\{p\}$ and $\bar t\in [pq]\setminus\{p\}$ such that,
for all $s\in[p\bar s]^\circ$ and $t\in [p\bar t]$,
\begin{equation}\label{eqn2.35}
|\bar ts|>|\tilde{\bar t}\tilde s|, |\bar st|>|\tilde{\bar s}\tilde t|, |\bar s\bar t|=|\tilde{\bar s}\tilde{\bar t}|  \text{ and } \angle \bar tp\bar s-\tilde\angle_k\bar tp\bar s=\angle qpr-\tilde\angle_k qpr,
\end{equation}
where $\tilde{\bar s},\tilde s,\tilde{\bar t},\tilde t$ belong to the comparison triangle
$\triangle\tilde p\tilde q\tilde r\subset\Bbb S^2_k$ of $\triangle qpr$ and correspond to $\bar s,s,\bar t,t$ respectively with
$\tilde {\bar s}\in [\tilde p\tilde r]$ and $|\tilde{\bar s}\tilde p|=|\bar sp|$, etc.

Our strategy is also to look for a $\triangle pqr$ with a `bad' angle $\angle qpr$
such that the corresponding $\triangle p\bar t\bar s$ is the desired triangle.
We will fulfill the task through the following four steps based on a general $\triangle pqr$ in which $\angle qpr$ is a `bad' angle.

(In Cases 1-3, for any $\triangle pqr$  with `bad' angle $\angle qpr$,
we can conclude that at least one of $\angle p\bar t\bar s$ and $\angle p\bar s\bar t$ is less than $\frac\pi2$,
so it suffices to find a triangle with an acute `bad' angle.
Unfortunately, as mentioned above, in the situation here
we cannot conclude such a property for a general triangle with a `bad' angle.)

\vskip1mm

\noindent {\bf Step 1}. To show that $\triangle qpr$ can be chosen to satisfy that there is at most one point $t\in [pr]^\circ$
such that $[qt]$ is perpendicular to $[pr]$.

\vskip1mm

Note that it suffices to consider the case where there are two distinct points
$t_1, t_2\in [pr]^\circ$ such that $\angle qt_ip=\angle qt_ir=\frac\pi2$.
Claim: up to a new choice, $t_i$ satisfies that either $\angle qtt_1=\angle qtt_2=\frac\pi2$ for all
$t\in [t_1t_2]^\circ$, or one of $\angle qtt_i$ ($i=1,2$) is less than $\frac\pi2$ for all
$t\in [t_1t_2]^\circ$. In fact, if there is $t\in [t_1t_2]^\circ$ such that $\angle qtt_1<\frac\pi2$,
then $\angle qt't_1<\frac\pi2$ for $t'$ sufficiently close to $t$ (by Lemma \ref{lem1.7}), which implies the claim.

If $\angle qtt_1=\angle qtt_2=\frac\pi2$ for all $t\in [t_1t_2]^\circ$, then it has to hold that $|qt|=|qt_1|=|qt_2|$ by Lemma \ref{lem1.4}.
Since $U_x$ can be sufficiently small, it follows that $\tilde\angle_k qtt_i<\frac\pi2$, which contradicts (\ref{eqn2.25}).

If $\angle qtt_1<\frac\pi2$ or $\angle qtt_2<\frac\pi2$ for all
$t\in [t_1t_2]^\circ$, $\triangle qpr$ can be chosen to be $\triangle qt_1t_2$. Note that
$\angle qt_1t_2=\angle qt_2t_1=\frac\pi2$, and at least one of $\tilde\angle_k qt_1t_2$ and $\tilde\angle_k qt_2t_1$
is less than $\frac\pi2$ as long as $U_x$ is small enough; i.e., $\triangle qt_1t_2$ has a `bad' angle.

\vskip1mm

\noindent {\bf Step 2}. To show that
there is $[\hat p\hat r]\subset[pr]^\circ$ such that  $\angle q\hat p\hat r$ is acute and `bad' in $\triangle q\hat p\hat r$.

\vskip1mm

Since $\angle qpr>\tilde\angle_k qpr$, by Lemma \ref{lem1.4}
there is corresponding $r'\in[pr]^\circ$ and $\tilde r'\in[\tilde p\tilde r]$ which satisfy (\ref{eqn2.26}) (cf. (\ref{eqn2.9})).
Then similarly we can apply  (2.2.1) to conclude that
$\triangle qr'p$ or $\triangle qr'r$ contains a `bad' angle $\angle qr'p$ or $\angle qr'r$ respectively
(cf. (\ref{eqn2.10})). Moreover, as in Case 1 (see the end of the proof for Case 1), up to repeating such a cutting operation finite times we can assume
that $|pr|\ll |pq|$ which implies that $\tilde\angle_kpqr\ll \frac\pi2$, and thus by (2.2.1) again we can conclude that
\begin{equation}\label{eqn2.380}
\text{$\angle qr'p$ and $\angle qr'r$ are `bad' angles in $\triangle qr'p$ and $\triangle qr'r$ respectively}.
\end{equation}
Since there is at most one point $t\in [pr]^\circ$ such that $\angle qtp=\frac\pi2$
(by Step 1), we can assume that in $[pr']^\circ$ or $[r'r]^\circ$, say $[r'r]^\circ$,
there is no point $t$ such that $\angle qtr'=\frac\pi2$.
Meantime, note that $\angle qtr'+\angle qtr=\pi$ for all $t\in[r'r]^\circ$ (by Lemma \ref{lem1.6}). Then by repeating the cutting operation on $\triangle qr'r$ two more times,
we can locate $[\hat p\hat r]\subset [r'r]\setminus\{r\}\subset [pr]^\circ$ such that $\angle q\hat p\hat r$ is acute and `bad' in $\triangle q\hat p\hat r$.

\vskip1mm

\noindent {\bf Step 3}.  To show that there is $\check r \in [\hat p \hat r]$ and $p_i\in [\hat p\check r]^\circ$ with $p_i\to \check r$ as $i\to\infty$ such that
$\angle qp_i\check r$ is acute and `bad' in $\triangle qp_i\check r$.

\vskip1mm

Let $\triangle \tilde q\tilde{\hat p}\tilde{\hat r}\subset\mathbb S_k^2$ be a comparison triangle of $\triangle q\hat p\hat r$.
Since $\angle q\hat p\hat r$ is `bad' (i.e. $\angle q\hat p\hat r>\angle \tilde q\tilde{\hat p}\tilde{\hat r}$),
by Lemma \ref{lem1.4} there is $t_1\in [\hat p\hat r]\setminus\{\hat p\}$ such that
$|qt_1|=|\tilde{q}\tilde t_1|$ and $|qt|>|\tilde{q}\tilde t|$ for all $t\in[\hat pt_1]^\circ$,
where $\tilde t_1,\tilde t\in [\tilde{\hat p}\tilde{\hat r}]$ with $|\tilde t_1\tilde{\hat p}|=|t_1\hat p|$
and $|\tilde t\tilde{\hat p}|=|t\hat p|$.
As a result, by (\ref{eqn2.25}), it is not hard to see that
\begin{equation}\label{eqn2.38}
 |qt|\neq|q[\hat pt_1]| \text{ for any } t\in [\hat pt_1]^\circ.
\end{equation}
And due to `$\angle q\hat p\hat r<\frac\pi2$' (see Step 2),
(\ref{eqn2.38}) has a more precise version
\begin{equation}\label{eqn2.39}
|qt|>|qt_1| \text{ for any } t\in [\hat pt_1]^\circ,
\end{equation}
which, by Lemma \ref{lem1.4}, implies that
\begin{equation}\label{eqn2.40}
\angle qt_1\hat p\geq \frac\pi2.
\end{equation}

Note that $\triangle \tilde{q}\tilde{\hat p}\tilde t_1$ is a comparison triangle of $\triangle q\hat pt_1$,
so $\angle q\hat p t_1$ is also `bad' in $\triangle q\hat p t_1$.
Together with `$\hat p\in [pr]^\circ$' (see Step 2) and by Lemma \ref{lem1.7},
this implies that $\angle qtt_1$ is `bad' in $\triangle qtt_1$ for $t\in [\hat p t_1]^\circ$ sufficiently close to $\hat p$.
Hence, one of the following two cases must happen:

$\bullet$ for all $t\in [\hat p t_1]^\circ$, $\angle qtt_1$ is `bad' in $\triangle qtt_1$;

$\bullet$ there is $t_\ast\in [\hat p t_1]^\circ$ such that $\angle qtt_1$ is `bad' in $\triangle qtt_1$
for all $t\in [\hat pt_\ast]$ except $t_*$.

In the former case, we can put $\check r=t_1$; otherwise, $\angle qtt_1\geq\frac\pi2$ for all $t\in [\hat p t_1]^\circ$ sufficiently close to $t_1$,
and thus $|qt|\leq |qt_1|$ by Lemma \ref{lem1.4}, which contradicts (\ref{eqn2.39}).

In the latter case, we shall put $\check r=t_*$ by showing that
there is  $p_i\in [\hat p t_*]^\circ$ with $p_i\to t_*$ as $i\to\infty$
such that $\angle qp_it_*$ is acute and `bad' in $\triangle qp_it_*$.

We first observe that, for any $t\in[\hat pt_*]\setminus\{t_*\}$,
at least one of $\angle qtt_*$ and $\angle qt_*t$ is a `bad' angle of $\triangle qtt_*$.
Otherwise, $\angle qtt_*\leq\tilde\angle_k qt t_*$
and $\angle qt_*t\leq\tilde\angle_kqt_*t$. Consider $\triangle \tilde q\tilde{t}\tilde t_*\subset\mathbb S^2_k$,
a comparison triangle of $\triangle qt t_*$, and let
$[\tilde{t}\tilde t_*]\subset [\tilde{t}\tilde z]\subset\mathbb S^2_k$
with $|\tilde{t}\tilde z|=|t t_1|$.
Since $\angle qt t_1>\tilde\angle_k qtt_1$ (by the badness of $\angle qt t_1$),
`$\angle qt t_*\leq\tilde\angle_k qt t_*$' implies that $|qt_1|<|\tilde q\tilde z|$.
Plus `$\angle qt_*t_1\leq\tilde\angle_k qt_*t_1$' (note that $\angle qt_*t_1$ is not `bad' in $\triangle qt_*t_1$),
we conclude that $\angle qt_*t_1<\angle \tilde q\tilde t_*\tilde z$,
and thus by Lemma \ref{lem1.6} we have that $\angle qt_*t>\angle \tilde q\tilde t_*\tilde{t}=\tilde\angle_k qt_*t$, a contradiction.

Based on the observation right above, for $t$ close to $t_*$,
we can conduct a cutting operation on $\triangle qtt_*$ as in Step 2 to locate
a $\hat t\in [tt_*]^\circ$ such that $\angle q\hat tt_*$ is `bad' in $\triangle q\hat tt_*$ (cf. (\ref{eqn2.380})).
Namely, we can locate  $p_i\in [\hat p t_*]^\circ$ with $p_i\to t_*$ as $i\to\infty$
such that $\angle qp_it_*$ is `bad' in $\triangle qp_it_*$.
On the other hand, we claim that $\angle qt_*t_1<\frac\pi2$,
which implies that $\angle qp_it_*<\frac\pi2$ (by Lemma \ref{lem1.7}).
In fact, the claim follows from that $\angle q t_*t_1\leq \tilde\angle_k qt_*t_1$
(note that $\angle q t_*t_1$ is not `bad' in $\triangle q t_*t_1$) and $\tilde\angle_k q t_*t_1<\frac\pi2$
(note that $|qt_*|>|qt_1|$ by (\ref{eqn2.39}) and $U_x$ is sufficiently small).

\vskip1mm

As shown in the beginning of the proof, for each $\triangle qp_i\check r$, there is
$\bar s_i\in [p_i\check r]\setminus\{p_i\}$ and $\bar t_i\in [p_iq]\setminus\{p_i\}$ such that
the corresponding (\ref{eqn2.35}) holds for all $s\in[p_i\bar s_i]^\circ$ and $t\in [p_i\bar t_i]$.

\vskip1mm

\noindent {\bf Step 4}.  To show that $\triangle p_i\bar t_i\bar s_i$ is our desired triangle for large $i$
(and thus the proof is done).

\vskip1mm

Note that if $|\bar s_i\bar t_i|\geq |\bar s_ip_i|$ or $|\bar s_i\bar t_i|\geq |\bar t_ip_i|$, say $|\bar s_i\bar t_i|\geq |\bar t_ip_i|$,
then by Lemma \ref{lem1.4} the acuteness of $\angle \bar t_ip_i\bar s_i\ (=\angle qp_i\check r)$
implies that $|\bar t_is|=|\bar t_i[p_i\bar s_i]|$ for some $s\in[p_i\bar s_i]^\circ$. Then,
by the inequalities in the corresponding (\ref{eqn2.35}) for $\bar s_i$ and $\bar t_i$, we can conclude that $\triangle p_i\bar t_i\bar s_i$ is our desired triangle.

Hence, in the rest of the proof, we only need to consider the case where
\begin{equation}\label{eqn2.41}
|\bar s_i\bar t_i|<|\bar s_ip_i| \text{ and } |\bar s_i\bar t_i|<|\bar t_ip_i|.
\end{equation}
In this case, it suffices to show that one of $\angle\bar s_i\bar t_ip_i$ and $\angle\bar t_i\bar s_ip_i$ is less than $\frac\pi2$ for large $i$,
which together with the acuteness of $\angle \bar t_ip_i\bar s_i$ also implies that $\triangle p_i\bar t_i\bar s_i$ is our desired triangle by Lemma \ref{lem1.4}.

The main tool here is (\ref{lem2.6}.2). In order to apply it, we need to verify its conditions in the situation here.
We first note that (\ref{eqn2.41}) with $|p_i\bar s_i|\leq|p_i\check r|\to 0$ as $i\to\infty$ implies that $|p_i\bar t_i|\to 0$ as $i\to\infty$,
and thus there is $r_i\in[p_iq]$ such that
\begin{equation}\label{eqn2.42}
|p_i\bar t_i|=\frac12|p_ir_i|.
\end{equation}
On the other hand, note that $[p_i\bar s_i]\subseteq[p_i\check r]\subset[\hat p\hat r]\subset[pr]^\circ$, so by Lemma \ref{lem1.7}
\begin{equation}\label{eqn2.43}
\angle\bar t_ip_i\bar s_i\ (=\angle qp_i\check r) \to \angle q\check rr \quad \text{as} \quad i\to\infty.
\end{equation}
Moreover, we can assume that $X$ is of curvature $\geq k_x$ around $x$,
so we have that $\lim\limits_{i\to\infty}(\angle\bar t_ip_i\bar s_i-\tilde\angle_{k_x}\bar t_ip_i\bar s_i)=0$ by (\ref{lem2.6}.1), which
together with (\ref{eqn2.41}) and (\ref{eqn2.43}) implies that
there is a $c>0$ such that
\begin{equation}\label{eqn2.44}
\min\{\tilde\angle_{k_x}\bar s_ip_i\bar t_i,\tilde\angle_{k_x}p_i\bar s_i\bar t_i,\tilde\angle_{k_x}p_i\bar t_i\bar s_i\}>c.
\end{equation}
Note that (\ref{eqn2.42}) and (\ref{eqn2.44}) enable us to apply (\ref{lem2.6}.2) on $\triangle \bar t_ip_i\bar s_i$ to
conclude that
$$|\angle\bar s_i\bar t_ip_i-\tilde\angle_{k_x}\bar s_i\bar t_ip_i|+|\angle\bar t_ip_i\bar s_i-\tilde\angle_{k_x}\bar t_ip_i\bar s_i|+
|\angle\bar t_i\bar s_ip_i-\tilde\angle_{k_x}\bar t_i\bar s_ip_i|\longrightarrow 0 \text{ as } i\longrightarrow\infty.
$$
Note that
$$
\tilde\angle_{k_x}\bar s_i\bar t_ip_i+\tilde\angle_{k_x}\bar t_ip_i\bar s_i+\tilde\angle_{k_x}\bar t_i\bar s_i p_i\to \pi  \text{ as } i\longrightarrow\infty.
$$
Then plus (\ref{eqn2.43}) we can conclude that at least one of $\angle\bar s_i\bar t_ip_i$ and $\angle\bar t_i\bar s_ip_i$ is less than $\frac\pi2$ for large $i$.
\end{proof}


\section{Proof of Theorem A for curvature $\equiv k$ on $X^\circ$}

In this section, we will show the rigidity part of Theorem A,
i.e., we will prove that $X^\circ$, the interior part of $X$, is a Riemannian manifold with sectional curvature equal to $k$ if  each $x\in X$ satisfies that
\begin{equation}\label{eqn3.1}
\text{for any $q\in U_x$ and $[r_1r_2]\subset U_x$, if there is $p\in [r_1r_2]^\circ$
	such that $|qp|=|q[r_1r_2]|$, then
	$\tilde\angle_kqpr_i=\frac\pi2$}.
\end{equation}

Note that, by the conclusion in Theorem A for curvature $\geq k$, (\ref{eqn3.1}) implies clearly that $X$ is of curvature $\geq k$ around
each $x\in X$ (i.e. $X$ is an Alexandrov space with curvature $\geq k$). Hence, at each $z\in X$, we can consider the space of directions and the tangent cone,
$\Sigma_zX$ and $C_zX$, which are still Alexandrov spaces of curvature $\geq 1$ and $\geq 0$ respectively. (Refer to Section 1 for $\Sigma_zX$ and $C_zX$.)
An easy observation from (\ref{eqn3.1}) is  that $\Sigma_zX$ and $C_zX$ also satisfy a corresponding property of (\ref{eqn3.1}), i.e. Lemma \ref{lem3.1} below.
This makes it possible to apply the inductive assumption on
$\Sigma_zX$ which is of dimension one less than $X$ and has an empty boundary if $z\in X^\circ$.

\begin{lemma}\label{lem3.1}
 Let $X$ be a complete Alexandrov space, and let $U_x$ be a neighborhood
of $x\in X$ satisfying {\rm (\ref{eqn3.1})}, and let $z\in U_x$. Then for all $\bar q\in C_zX$  (resp. $\in\Sigma_zX$)
and $[\bar r_1\bar r_2]\subset C_zX$ (resp. $\subset\Sigma_zX$),
\begin{equation}\label{eqn3.2}
\text{if there is $\bar p\in [\bar r_1\bar r_2]^\circ$
such that $|\bar q\bar p|=|\bar q[\bar r_1\bar r_2]|$, then
$\tilde\angle_{0\ (resp.\ 1)}\bar q\bar p\bar r_i=\frac\pi2$}.
\end{equation}
\end{lemma}

\begin{proof}	
By definition, $C_zX$ is the cone over $\Sigma_zX$ (\cite{BGP})
\footnote{The metric on $C_zX$ is defined from the Law of Cosine on $\mathbb R^2$  by viewing distances on $\Sigma_zX$
as angles (\cite{BGP}).}. So, it is not hard to see that the property of (\ref{eqn3.2}) for $C_zX$ implies that for $\Sigma_zX$.
In order to see (\ref{eqn3.2}) for $C_zX$, one just needs to notice that $C_zX$ is the limit space of
$(\lambda X, z)$ as $\lambda\to+\infty$ by Lemma \ref{lem1.8}.
\end{proof}

Note that (\ref{eqn3.1}) is contained in (\ref{eqn2.25}). Then, since $X$ is of curvature $\geq k$ around $x$,
by Lemma \ref{lem2.4} we have another easy observation:

\begin{lemma} \label{lem3.2}
Let $X$ be a cone over a circle with perimeter less than $2\pi$ (an Alexandrov space
of curvature $\geq 0$). Then around its vertex there is no neighborhood such that {\rm(\ref{eqn3.1})} holds with respect to $k=0$.
\end{lemma}

\begin{proof}
The lemma is an easy corollary of Lemma \ref{lem2.4} because
we can find two points sufficiently close to the vertex of $X$ between which
there are two minimal geodesics.
\end{proof}

To $q, p, [r_1r_2]$ in (\ref{eqn3.1}), we associate $\tilde{q}\in\mathbb S_k^2$ and $\tilde{p}\in[\tilde{r}_1\tilde{r}_2]\subset\mathbb S_k^2$ with
$|\tilde r_1\tilde r_2|=|r_1r_2|$, $|\tilde q\tilde r_i|=|qr_i|$ and $|\tilde r_i\tilde p|=|r_ip|$.
Since $X$ is of curvature $\geq k$ on $U_x$, we know that $|qp|\geq |\tilde q\tilde p|$ (by Definition \ref{dfn1.2}) and
any $[qp]$ is perpendicular to $[r_1r_2]$ at $p$ (by Lemma \ref{lem1.6}). Then we have the third easy observation from (\ref{eqn3.1}):
\begin{equation}\label{eqn3.3}
\text{$|qp|=|\tilde q\tilde p|$ (and $[\tilde q\tilde p]$ has to be perpendicular to $[\tilde r_1\tilde r_2]$)}.
\end{equation}
Thereby, the following rigidity version of Theorem \ref{thm1.3} can be applied.

\begin{theorem}[\cite{GM}] \label{thm3.3}
	 Let $X$ be a complete Alexandrov space with curvature $\geq k$,
	 and let $q\in X$ and  $[pr]\subset X$. If there is $s\in [pr]^\circ$ such that $|qs|=|\tilde q\tilde s|$,
	 where $\tilde q\in\mathbb S^2_k$ and $\tilde s\in[\tilde p\tilde r]\subset \mathbb S^2_k$ with $|\tilde q\tilde p|=|qp|$, $|\tilde p\tilde r|=|pr|$, $|\tilde q\tilde r|=|qr|$ and $|ps|=|\tilde p\tilde s|$,
	 then there is $[pq]$ and $[qr]$ such that $[pr]$ together with them bounds a convex surface which can be embedded isometrically into $\mathbb S^2_k$.
\end{theorem}

Based on the three observations above, we can prove the rigidity part of Theorem A.

\begin{proof}[Proof for the rigidity part of Theorem A]\

Note that our assumption is that (\ref{eqn3.1}) holds for each $x\in X$.
So, as mentioned above, $X$ is an Alexandrov space with curvature $\geq k$ and $\dim(X)$, the dimension of $X$, can be defined.
We will give a proof by induction on $\dim(X)$.

We first consider the case where $\dim(X)=2$.
It suffices to show that, for any $p\in X^\circ$, there is a convex surface
$\mathcal{S}\ (\subseteq X)$ which can be embedded isometrically into $\mathbb S^2_k$ such that $p\in \mathcal{S}^\circ$.

Since $p\in X^\circ$,
$\Sigma_pX$ is a circle with perimeter $\leq 2\pi$ (\cite{BGP}). Since $C_pX$ is the cone over $\Sigma_pX$,
by Lemma \ref{lem3.2}, the perimeter of $\Sigma_pX$ must be equal to $2\pi$,
and thus  $C_pX$ is a plane.  Let $\xi_i\in \Sigma_pX$,
$i=1,2,3$, be $\frac{2\pi}3$-separated (i.e. $|\xi_i\xi_j|=\frac{2\pi}3$ for $i\neq j$).
We know that there are $\{q_{il}\}_{l=1}^{\infty}\subset X$ with $|q_{1l}p|=|q_{2l}p|=|q_{3l}p|\to 0$ and the directions
$\uparrow_{p}^{q_{il}}\to \xi_i$ as $l\to\infty$ (\cite{BGP}).
On the other hand, we know that $C_pX$ is the limit space of
$(\frac{1}{|pq_{il}|}X, p)$ as $l\to\infty$ by Lemma \ref{lem1.8}. It follows that $q_{il}\to q_i$
as  $(\frac{1}{|pq_{il}|}X, p)\to C_pX$ with $|pq_i|=1$ and $\uparrow_p^{q_i}=\xi_i$ in $C_pX$.
Note that as $(\frac{1}{|pq_{il}|}X, p)\to C_pX$ any triangle $\triangle q_{1l}q_{2l}q_{3l}\to \triangle q_1q_2q_3$,
an equilateral triangle.  It then is easy to see that, for sufficiently large $l$,
$|q_{3l}[q_{1l}q_{2l}]|=|q_{3l}q|$ for some $q\in [q_{1l}q_{2l}]^\circ$.
Then by (\ref{eqn3.3}) and Theorem \ref{thm3.3}, $\triangle q_{1l}q_{2l}q_{3l}$ bounds
a convex surface $\mathcal{S}_l$ which can be embedded isometrically into $\mathbb S^2_k$
(note that there is a unique minimal geodesic between any two points around $p$ by Lemma \ref{lem2.4}).

We claim that $\mathcal{S}_l$ for large $l$ is just our desired surface.
In fact, $\mathcal{S}_l$ converges to the domain bounded by $\triangle q_1q_2q_3$ in the plane $C_pX$
as  $(\frac{1}{|pq_{il}|}X, p)\to C_pX$. It follows that, for large $l$, there is $\bar p\in\mathcal{S}_l^\circ$ such that $|p\bar p|=|p\mathcal{S}_l|$.
Since $\mathcal{S}_l$ is a convex surface in $X$ and $\dim(X)=2$, it must hold that $p=\bar p$, i.e. $p\in\mathcal{S}_l^\circ$.

\vskip1mm

We now assume that $\dim(X)=n\geq 3$. Similarly, for any $p\in X^\circ$,
it suffices to show that $p$ lies in the interior part of a convex domain (in $X$)
which can be embedded isometrically into $\mathbb S^n_k$.

We first show that $C_pX$, the cone over $\Sigma_pX$, is isometric to the Euclidean space $\mathbb R^n$,
i.e.  $\Sigma_pX$ is isometric to the unit sphere $\mathbb S_{1}^{n-1}$.
Note that $\Sigma_pX$ is a complete Alexandrov space with curvature $\geq 1$,
and $(\Sigma_pX)^\circ=\Sigma_pX$ with dimension equal to $n-1$ because $p\in X^\circ$.
By Lemma \ref{lem3.1}, we can apply the inductive assumption on $\Sigma_pX$ to conclude that
it is a complete Rimannian manifold with sectional curvature equal to 1.
Then we need only to show that $\Sigma_pX$ is simply connected.
If it is not true, then by classical results in Riemannian geometry
there is a closed geodesic, a circle $S^1\subset\Sigma_pX$, with perimeter less than $2\pi$.
Note that the cone over the $S^1$ is convex in $C_pX$, which is impossible by Lemma \ref{lem3.2}.

Similarly, since $\Sigma_pX$ is isometric to $\mathbb S_1^{n-1}$, we can select an $(\arccos\frac{-1}{n})$-separated subset $\{\xi_i\in \Sigma_pX| i=1,2,\cdots, n+1\}$,
and $\{q_{il}\}_{l=1}^\infty\subset X$ with $|q_{1l}p|=|q_{2l}p|=\cdots=|q_{(n+1)l}p|\to 0$ and the directions
$\uparrow_{p}^{q_{il}}\to \xi_i$ as $l\to\infty$.
And it follows that $q_{il}\to q_i$
as  $(\frac{1}{|pq_{il}|}X, p)\to C_pX$ with $|pq_i|=1$ and $\uparrow_p^{q_i}=\xi_i$.

Claim: {\it For sufficiently large $l$, there is a convex simplex $\Delta_l$ with vertices $q_{1l},\cdots, q_{(n+1)l}$
which can be embedded isometrically into $\mathbb S^n_k$}.
By the claim, we need only to show that $p\in\Delta_l^\circ$ for large $l$.
Note that as $(\frac{1}{|pq_{il}|}X, p)\to C_pX$, $\Delta_l$
converge to the simplex with vertices $q_1,\cdots, q_{n+1}$ in $C_pX\stackrel{\rm iso}{\cong}\mathbb R^n$.
It follows that, for large $l$, there is $\bar p\in\Delta_l^\circ$  such that $|p\bar p|=|p\Delta_l|$.
Since $\Delta_l$ is a convex in $X$ and $\dim(\Delta_l)=\dim(X)$, it must hold that $p=\bar p$, i.e. $p\in\Delta_l^\circ$.

To complete the proof, it suffices to verify the claim right above.
Note that there is a unique minimal geodesic between any two points around $p$ by (\ref{eqn3.1}) and Lemma \ref{lem2.4}.
And as in the proof for $\dim(X)=2$, for large $l$, $|q_{3l}[q_{1l}q_{2l}]|=|q_{3l}q|$ for some $q\in [q_{1l}q_{2l}]^\circ$.
So, the triangle $\triangle q_{1l}q_{2l}q_{3l}$ bounds
a convex surface $\mathcal{S}_l$ which can be embedded isometrically into $\mathbb S^2_k$ (by (\ref{eqn3.3}) and Theorem \ref{thm3.3}).
Furthermore, there is $\bar{q}_{4l}\in\mathcal{S}_l^\circ$
(around the center of $\mathcal{S}_l$) such that $|q_{4l}\bar q_{4l}|=|q_{4l}\mathcal{S}_l|$.
Let $\tilde q_{il}\in\mathbb S^3_k$, $i=1,2,3,4$, be the vertices of a simplex of dimension 3 with
$|\tilde q_{il}\tilde q_{jl}|=|q_{il}q_{jl}|$.
By applying (\ref{eqn3.3}) and Theorem \ref{thm3.3} iteratively, it is not hard to see that $\bigcup_{s\in\mathcal S_l}[q_{4l}s]$ is convex and isometric to the simplex
in $\mathbb S^3_k$ with vertices  $\tilde q_{1l}, \tilde q_{2l}, \tilde q_{3l}, \tilde q_{4l}$.
And step by step, we can eventually get the desired $\Delta_l$ in the claim.
\end{proof}


\section{Proofs of Theorem C and Corollary D}

\begin{proof}[Proof of Theorem C]\

Let $x$ be an interior point in a complete Alexandrov space $X$ with curvature $\geq k$.
As mentioned in Section 3, we can consider the space of directions and the tangent cone at $x$, $\Sigma_xX$ and $C_xX$.
And by Lemma \ref{lem1.8}, we know that $C_xX$ is the limit space of
$(\lambda X, x)$ as $\lambda\to+\infty$. It is not hard to see that if we substitute the condition
``$\left|\frac{|qr|^2}{|pq|^2+|pr|^2}-1\right|< \chi(\varepsilon)$ for all $[pq],[pr]\subset B_x(\varepsilon)$ with $\angle qpr=\frac\pi2$
and $\chi(\varepsilon)\to 0$ as $\epsilon\to0$'' of Theorem C for the condition of Lemma \ref{lem3.1}, the conclusion of Lemma \ref{lem3.1} for $\Sigma_xX$ and $C_xX$ still holds.
Then from the proof for the rigidity part of Theorem A (in Section 3), we can conclude that
$\Sigma_xX$ is isometric to the unit sphere, i.e. $x$ is Riemannian point.
\end{proof}

\begin{proof}[Proof of Corollary D]\

Let $x$  be an interior point in a complete CBB-type Alexandrov space $X$.
If  $x$ is also a CBA-type point, then by Theorems \ref{thm1.3} and 1.3$'$ it is easy to see that
there is a function $\chi(\varepsilon)$ with  $\chi(\varepsilon)\to 0$ as $\varepsilon\to0$
such that for all $[pq],[pr]\subset B_x(\varepsilon)$ with $\angle qpr=\frac\pi2$
$$\left|\frac{|qr|^2}{|pq|^2+|pr|^2}-1\right|< \chi(\varepsilon).$$
So, by Theorem C, $x$ is Riemannian point.
\end{proof}

\noindent School of Mathematical Sciences (and Lab. math. Com.
Sys.), Beijing Normal University, Beijing, 100875
P.R.C.\\
e-mail: suxiaole$@$bnu.edu.cn; wyusheng$@$bnu.edu.cn

\vskip2mm

\noindent Mathematics Department, Capital Normal University,
Beijing, 100037 P.R.C.\\
e-mail: 5598@cnu.edu.cn

\end{document}